\documentclass[a4paper]{amsart}
\usepackage{latexsym}\usepackage{ifthen}\usepackage[leqno]{amsmath}
\usepackage{enumerate}\usepackage{calc}\usepackage{hyphenat}
\usepackage{amstext,amsbsy,amsopn,amsthm,amsgen}
\usepackage{amsfonts,amscd,amsxtra,upref}
\usepackage{graphicx}
\usepackage{caption}
\usepackage{float}
\restylefloat{table}
\usepackage{array}

\usepackage{epstopdf}
\usepackage{dsfont}
\usepackage{hyperref}
\usepackage{mathrsfs}\usepackage{euscript}\usepackage{amssymb}
\swapnumbers
\theoremstyle{plain}

\makeatother
\newtheorem{Thm}{Theorem}[section]

\newtheorem{Pro}[Thm]{Proposition}
\newtheorem*{MConj}{Modularity conjecture}

\theoremstyle{definition}

\theoremstyle{remark}
\newtheorem{Rem}[Thm]{Remark}
\numberwithin{equation}{section}

\newcommand{\ITE}[3]{\ifthenelse{#1}{#2}{#3}}\newcommand{\ITEE}[4][]{\ITE{\equal{#2}{#3}}{#4}{#1}}
 \ITE{\isundefined{\texorpdfstring}}{\newcommand{\texorpdfstring}[2]{#1}}{}

\newcommand{\myData}[1][]{
 \author[D.\ Burek]{Dominik Burek}
 \address{\ITEE{#1}{*}{D.\ Burek{}\\{}}%%
  Instytut Matematyki\\{}Wydzia\l{} Matematyki i~Informatyki\\{}Uniwersytet Jagiello\'{n}ski\\{}%%
  ul.\ \L{}ojasiewicza 6\\{}30-348 Krak\'{o}w\\{}Poland}
 \email{dominik.burek@student.uj.edu.pl}
 }

\newcommand{\PP}{\mathbb{P}}

\begin{document}
\title{Rigid realizations of modular forms in Calabi--Yau threefolds}
\myData

\begin{abstract}
We construct examples of modular rigid Calabi--Yau threefolds, which give a realization of some new weight 4 cusp forms.
\end{abstract}

%\thanks{This work was partially supported by the grant 346300 for IMPAN from the Simons Foundation and the matching 2015-2019 Polish MNiSW fund.}
\maketitle

\section{Introduction}

The aim of this note is to construct new rigid Calabi--Yau realizations of weight 4 cusp forms as a resolution of a quotient of non--rigid examples listed by C.\ Mayer in \cite{Meyer}.

Recall that a smooth, projective threefold is called a \emph{Calabi--Yau} threefold if it satisfies the following two conditions:
\begin{enumerate}[\upshape(i)]
\item The canonical bundle of $X$ is trivial and
\item $H^{1}(X,\mathcal{O}_{X})=H^{2}(X,\mathcal{O}_{X})=0$.
\end{enumerate} 
A Calabi--Yau threefold is said to be \emph{rigid} if it has no infinitesimal deformations i.e. $H^{1}(X,\mathcal{T}_{X})=0.$
By Serre duality it is equivalent to the vanishing of $H^{2}(X,\Omega_{X})$ or $h^{2,1}(X)=0.$

A lot of research on Calabi--Yau threefolds was motivated by the following conjecture:

\begin{MConj} Any rigid Calabi--Yau threefold $X$ defined over $\mathbb{Q}$ is modular i.e. its $L$--series $L(X, s)$
coincides up to finitely many Euler factors with the Mellin transform $L(f, s)$ of a
modular cusp form $f$ of weight $4$ with respect to $\Gamma_{0}(N),$ where the level $N$ is only
divisible by the primes of bad reduction. \end{MConj}

In \cite{DM} and \cite{DMimprove} L.\ Dieulefait and J.\ Manoharmayum proved the modularity conjecture for all rigid Calabi--Yau threefold defined over $\mathbb{Q}$ satisfying some mild assumptions on primes of bad reduction. 

Finally the conjecture was obtained in \cite{YuiG} from the proof of Serre's conjecture given by C.\ Khare and J.-P.\ Wintenberger; see \cite{KhareWint1} and \cite{KhareWint2}.
However the question asked by B.\ Mazur and D.\ van Straten --- which modular forms can be realized as a modular form of some rigid Calabi--Yau threefold, is still open. 
%However an explicit correspondence between rigid Calabi--Yau threefolds and weight 4 cusp forms is still rare and ubiquitous.\par

\subsection*{Acknowledgments} This paper is a part of author's master thesis. I am deeply grateful to my advisor S\l{}awomir Cynk for recommending me to learn this area and his help.

\section{Involutions of double octic Calabi--Yau threefolds}

Let $S\subset \PP^{3}$ be an arrangement of eight planes such that no
six intersect and no four contain a line.
Then there exists a resolution of singularities $X$ of the double
covering of $\PP^{3}$ branched along $S$ which is a Calabi--Yau 
threefold.

A crepant resolution described in \cite{CSz} is not uniquely determined but
depends on the order of double lines. We can avoid this difficulty
replacing the blow--up of all double lines separately with the blow--up
of their sum in the singular double cover. Now, by the universal
property of a blow--up, every automorphism $\phi$ of the projective space
$\PP^{3}$ lifts to an automorphism $\widetilde\phi$ of the double octic
$X$.

Any automorphism of a Calabi--Yau threefold acts on a canonical form $\omega$ of
$X$ by a multiplication with a scalar $\omega\mapsto \mu\omega$.
If the automorphism has finite order $n$ then $\mu^{n}=1$. We shall
call an automorphism \emph{symplectic} if it preserves the canonical
form $\omega$ i.e. $\mu=1$.

We shall investigate in more details the case of an automorphism 
of order two. The fixed locus $\operatorname{Fix}(f)$ of a symplectic
involution of a Calabi--Yau threefold is a disjoint union of smooth
curves, while the fixed locus of a non--symplectic involution is a
disjoint union of smooth surfaces and isolated points. 

\begin{Pro}[]{}\label{main} Let $\phi$ be an automorphism of the
  projective space $\PP^{3}$ of order two such that 
\begin{enumerate}[\upshape(i)]
\item the fixed locus of $\phi$ contains no double nor triple line of $S,$
\item\label{fourplanes} planes intersecting in fourfold point are not invariant by $\phi$,
\item a fixed line of $\phi$ intersects $S$ with odd multiplicity in
  at most two points.
\end{enumerate} Then the fixed locus of the induced involution $\widetilde\phi\colon X\to X$
contains no curve with positive genus.\end{Pro}

\begin{proof}
  Our strategy is to perform the resolution and after each step
  verify that the fixed locus of the lifting of $\phi$ to the partial
  resolution contains no irrational curves.
  We start with the (singular) double covering of $\mathbb P^{3}$
  branched along $S$. Any fixed curve of the involution must be mapped
  to a fixed line of $\phi$ in $\mathbb P^{3}$, if this curve  is
  irrational then, by the Riemann--Hurwitz formula, the line intersects
  $S$ with odd multiplicity in at most four points.
  This contradicts the second assumption, which in fact is necessary.

  The next step of the resolution is the blowing up of all fivefold
  points in the base of double covering, new branch divisor is the strict transform
  of $S$ plus the exceptional divisors. So any fixed curve of the
  involution is either birational to a fixed line in previous
  step or is a fixed line of the induced involution on a projective
  plane (isomorphic to an exceptional plane of the blow--up).
  At this step we introduce new double lines: intersections of the
  exceptional locus with strict transforms of the five planes.
  
  The involution is not the identity on the $\mathbb P^{2}$ because
  otherwise we will get a fixed divisor in the double octic.
  If one of the double lines is fixed (dashed line $l_{1}$ in the picture below), then the fixed locus on $\mathbb
  P^{2}$ consists of $l_{1}$ and an isolated point. On the other hand
  on each line $l_{2}, \ldots, l_{4}$ we have at least two fixed points, hence the other four lines
  intersect, which is impossible. 
  \begin{center}\includegraphics[width=5cm]{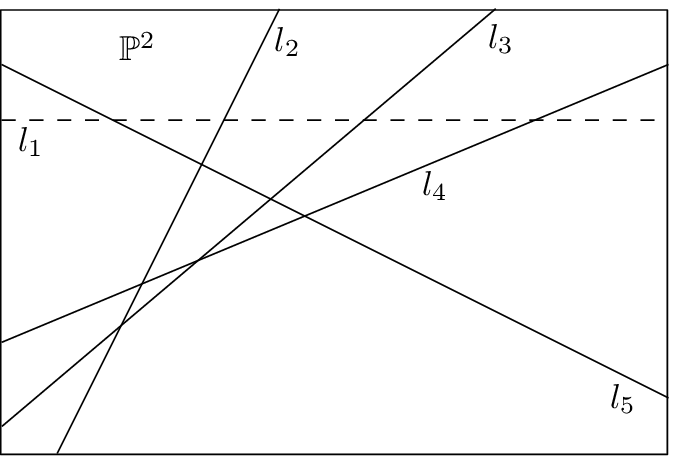}\end{center}\par

In the next step of resolution we blow up triple lines and as in the
case of a fivefold point we add the exceptional divisors to the branch divisor. In the
exceptional divisor we get the following
configuration: 
\begin{center}\includegraphics[width=5cm]{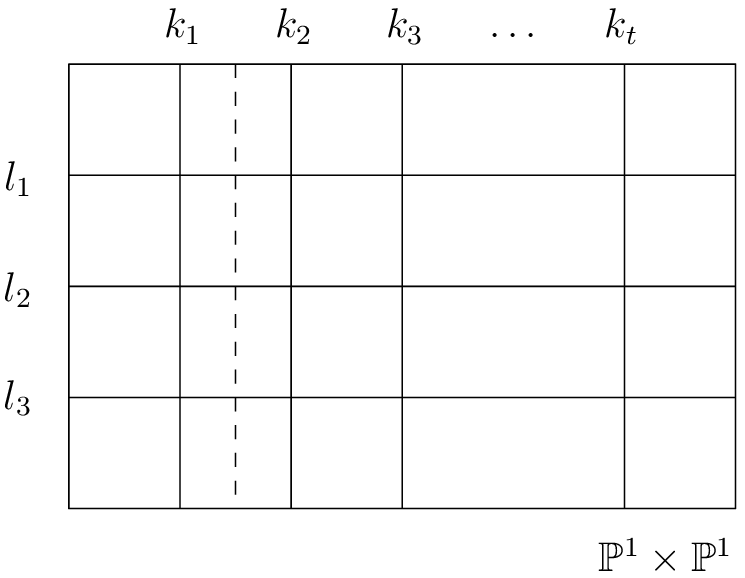}\end{center}
where $t$ denotes the number of fourfold points on that triple line.

Consequently any fixed curve is a vertical line (dashed line in the
picture) and so it is rational. Again we introduce new double lines: $l_{1},$ $l_{2},$ $l_{3}$ and $k_{1},$ $k_{2},$ $\ldots,$ $k_{t}$ (solid lines in the picture). Suppose that the line $k_{i}$ is fixed. 
Then we get four invariant planes: three planes containing $k_{i}$ and one transversal to them, which leads to the contradiction with assumption \hyperref[fourplanes]{(ii)}.

Now, we blow--up all fourfold points, this time we do not add the
exceptional divisors to the branch divisor. 
In the exceptional divisor over a fourfold point $P$ we get the following configuration:
\begin{center}\includegraphics[width=5cm]{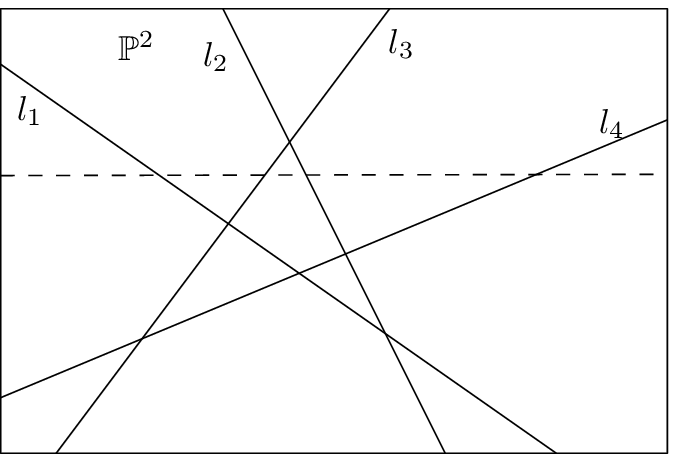}\end{center}
and in the double cover we get a double covering of $\mathbb P^{2}$
branched along the four solid lines. Any irrational fixed curve of the
involution maps to a fixed line of involution on $\mathbb P^{2}$ (a
dashed line in the picture), it is not of the other lines
$l_{1},\dots,l_{4}$ because over each of them the double cover is an
isomorphism.

In particular the involution on $\mathbb P^{2}$ is not
identity, its fixed locus consists of  the dashed line and an isolated
point. As an involution of a projective line has two fixed
points the  lines $l_{1},\dots,l_{4}$ intersect and consequently the
four planes passing through the fourfold point intersect along a line,
which contradicts our assumptions. 

Now the singular locus of the branch divisor is a union of double lines, some
of them were introduced in the previous step, but none of them is
fixed by the involution.   
In the last step we blow--up the sum of the double lines in the (singular)
cover, a fiber of the blow--up is a line at a double point or a sum
of three concurrent lines at a triple points. Since every fixed line is
birational to a fixed line from the previous step of the resolution or
a line (component of a fiber) the assertion of the proposition follows.

\end{proof}

\section{Main result}
In \cite{Meyer} C. Meyer gave a list of 63 one--parameter families of double octics $\{X_{\tau}\}_{\tau\in\PP^1}$ with $h^{2,1}(X_{\tau})=1$. Using an extensive computer search he found 18 examples in 11 families (see \cite{Meyer}, table on page 54) such that for all primes $p\le97$ the trace of the Frobenius map $\operatorname{Frob}_p$ equals $a_p+pb_p$, where $a_p$ and $b_p$ are the coefficients of cusp forms $f_4$ and $f_2$ of weight 4 and 2, respectively. This gives a strong numerical evidence of the modularity of $X$ in the following sense: the semi-simplification of the Galois representation on $H^3(X)$ decomposes into two-dimensional pieces isomorphic to Galois representations associated to $f_4$ and $f_2$. 
Equivalently, the $L$ series of $X$ factors as $L(X,s)=L(f_4,s)L(f_2,s-1)$ (see \cite{HulekV1} and \cite{HulekV2}).
In fact modularity of all examples except Arr. no. 154 was proved in \cite{CM}. 

\begin{Rem}\label{remark} Different elements from Meyer's table with the same arrangement type are in fact isomorphic with a quadratic twist compatible with twists $\lambda$ in table \cite{Meyer}. For an Arr. no. 4, from \cite{CKC} we find that the map
$$\left( \begin{matrix} x \\ y \\ z \\ t \\ u \end{matrix} \right) \mapsto \left( \begin{matrix} ABy+ABz \\ -ABy \\ \left( -{A}^{2}+AB \right) x+ \left( -{A}^{2}+AB \right)y-{A}^{2}z+ \left({A}^{2}-AB \right) t \\ \left( AB-{B}^{2} \right) t \\ A^3B^3(A-B)^2u\end{matrix} \right)$$
which gives an isomorphism (over $\mathbb{Q}$) between $X_{(A:B)}$ and the quadratic twist of $X_{(A-B:A)}$ by $-A(A-B)$. In particular $X_{(1:-1)}$ and $X_{(1:2)}$ are isomorphic, they are also isomorphic to the quadratic twist of $X_{(2:1)}$ by $-2.$ Since the cusp forms 32k4A1 and 32A1 have CM by $\sqrt{-1}$ the quadratic twists by $2$ and $-2$ coincide. Similar transformations exists also for Arr. no 13, 249 and 267.
 \end{Rem}

\begin{Thm}\label{main2} Let $X$ be a Calabi--Yau threefold from Meyer's table corresponding to Arr. no. 4, 13, 21, 53, 244, 267, 274. Then there exists a symplectic involution $\phi$ on $X$ and a crepant resolution $Y$ of a quotient $X/\phi,$ which is a rigid, modular Calabi--Yau threefold.\end{Thm}

\begin{proof} By the remark \ref{remark} for each arrangement type in the theorem we can consider only one example. For all examples, we shall find a transformation $g\colon \PP^{1}\to \PP^{1}$, which has an isolated fixed point $\tau_{0}$ such that $g'(\tau_{0})=-1$ and isomorphisms $\widetilde{\phi_{\tau}}\colon X_{\tau}\to X_{g(\tau)}$ for which $\widetilde{\phi_{\tau_{0}}}$ is a symplectic involution satisfying conditions of the Proposition \ref{main}. 

In the table \hyperref[Tab]{Tab.1} we give an involution $\phi_{\tau}$ and the transformation $g$.

\begin{table}[h!]
\centering
\def\arraystretch{1.7}
\resizebox{\columnwidth}{!}{%
\begin{tabular}{c|c|c|l|cl}
\hline
Arr. no. & $g$ & $\tau_{0}$ & \multicolumn{1}{c|}{$\phi_{\tau}(x,y,z,t,u)$} & Cusp form  \\
\hline\hline
4 & $1/\tau$ & $-1$ & $\begin{aligned}\big{(}Ay+Az, -Ay, Ax + Ay, Bt, A^3Bu\big{)}\end{aligned}$ &  32k4A1 \\
\hline
13 & $1/\tau$ & $1$ & $\big{(}Bz, By, Bx, At, -AB^3u\big{)}$ &  32k4A1 \\
\hline
21 & $(-1-\tau)/\tau$ & $-1/2$ & $\begin{aligned}&\big{(}(A+B)x+(A+B)y, (-A-B)y,\\ &(-A-B)z, Ax+(A+B)y+Bt,\\ &(A+B)^3Bu\big{)} \end{aligned}$ &  32k4B1  \\
\hline
53 & $1/\tau$ & $1$ & $\left(Ay, Ax, -Bt, -Bz, A^2B^2u\right)$ &  32k4B1  \\
\hline
244 & $1/\tau$ & $-1$ & $\left(Bx, By, At, Az, -B^2A^2u\right)$ &  12k4A1  \\
\hline
267 & $1/\tau$ & $-1$ & $(t,-z,-y,x,u)$ &  96k4B1 \\
\hline
274 & $1/\tau$& $1$ & $\left(z,-t,x,-y,u\right)$ &  96k4E1 \\
\hline
\end{tabular}
\caption{Tab. 1}\label{Tab}
}
\end{table}

Since $\tau_{0}$ is a fixed point of $g$ the weighted projective transformation $\phi_{\tau_0}$ induces an automorphism $\widetilde{\phi_{\tau_0}}$ of $X_{\tau_0}$ that transform $\omega$ to $\lambda \omega$ where $\lambda$ is the quotient of the determinant of the transformation on $\PP^3$ by the coefficient in front of $u$. For each of the cases in the table \hyperref[Tab]{Tab.1} we compute $\lambda=1$, so the involution $\widetilde{\phi_{\tau_0}}$ is symplectic.

The action induced by $\widetilde{\phi_{\tau_0}}$ on $H^{1,2}(X_{\tau_0})$ is given, via the Kodaira--Spencer map, by the multiplication by $g'(\tau_{0}),$ thus $H^{1,2}(X_{\tau_{0}})^{\langle \widetilde{\phi_{\tau_{0}}}\rangle}=0.$ The quotient $X_{\tau_{0}}/\phi_{\tau_{0}}$ admits a crepant resolution $Y_{\tau_{0}}$ by blowing--up of double curves. From the special case of an orbifold formula (see \cite{CR}) we have the following formula
$$H^{1,2}(Y_{\tau_{0}})\simeq H^{1,2}(X_{\tau_{0}})^{\langle \phi_{\tau_{0}}\rangle} \oplus \bigoplus_{C\in \operatorname{Fix}(\phi_{\tau_{0}})} H^{0,1}(C),$$ where $\operatorname{Fix}(\phi_{\tau_{0}})$ consists of curves, which are fixed by the involution $\widetilde{\phi_{\tau_{0}}}.$ Proposition \ref{main} yields that $Y_{\tau_{0}}$ is a rigid Calabi--Yau manifold defined over $\mathbb{Q}.$
From the involution we get a monomorphism $H^3(Y_{\tau_0})\to H^3(X_{\tau_0})$, so the modularity (and the cusp form) of $Y_{\tau_{0}}$ follows from the modularity of $X_{\tau_{0}}$.\end{proof}

\begin{Rem} According to our knowledge the last two examples in \ref{main2} give first rigid Calabi--Yau realizations of modular forms 96k4B1 and 96k4E1.\end{Rem}

\end{document}